\documentclass{article}

\usepackage{amscd,amsxtra,amsthm,amsmath}
\usepackage[all]{xy}
\usepackage{etex}
\usepackage{pictex}
\usepackage{graphicx}
\usepackage{mathtools}
\usepackage{float}
\usepackage{xcolor}
\usepackage[utf8]{inputenc}
\usepackage[T1]{fontenc}

\usepackage{amsfonts}

\usepackage{multirow}
\newcommand{\ceil}[1]{\left \lceil #1 \right \rceil}
\newcommand{\floor}[1]{\left \lfloor #1 \right \rfloor}
\usepackage{comment}

\newtheorem{theorem}{Theorem}

\theoremstyle{definition}
\newtheorem{definition}[theorem]{Definition}

\newtheorem{remark}[theorem]{Remark}

\theoremstyle{corollary}

\theoremstyle{conjecture}

\theoremstyle{proposition}

\newcommand\blfootnote[1]{%
  \begingroup
  \renewcommand\thefootnote{}\footnote{#1}%
  \addtocounter{footnote}{-1}%
  \endgroup
}

\newcommand{\Mod}[1]{\ (\mathrm{mod}\ #1)}

\begin{document}
\parskip0pt
\parindent15pt
\baselineskip15pt    

\begin{center}
{\bf Rainbow Numbers for the Generalized Schur Equation $x_1 + x_2 +  \ldots + x_{m-1} = x_m$}
\vskip 20pt
{\bf Mark Budden}\\
 Department of Mathematics and Computer Science, Western Carolina University, Cullowhee, North Carolina\\
{\tt mrbudden@email.wcu.edu}\\
\vskip 10pt
{\bf Bruce Landman}\\
 Department of Mathematics, University of Georgia, Athens, Georgia\\
{\tt Bruce.Landman@uga.edu}\\
\end{center}

\vskip 30pt


\blfootnote{AMS Subject Classification (2020): 05D10, 11B75.}
\blfootnote{Keywords:  Schur numbers, rainbow Ramsey theory.}

\begin{abstract}
\noindent We consider the rainbow Schur number $RS_m(n)$, defined to be the minimum number of colors such that every coloring of $\{1,2,\ldots,n\}$, using all $RS_m(n)$ colors, contains a rainbow solution to the equation $x_1+x_2+\cdots +x_{m-1}=x_m$. Recently, the exact values of $RS_3(n)$ and $RS_4(n)$ were determined for all $n$. In this paper, we expand upon this work by providing a formula for $RS_m(n)$ that holds for all $m \geq 4$ and all $n$. A weakened version of the rainbow Schur number is also considered, for which one seeks solutions to the above-mentioned linear equation where, for a fixed $t \leq m$,  at least $t$ colors are used.
 \end{abstract}

\section{Introduction}

Many classical problems in Ramsey theory involve  determining the existence, or non-existence, of certain monochromatic structures under finite colorings of a set.  The subject known as rainbow Ramsey theory deals with what might be considered the opposite notion; namely, instead of looking for a structure such that all of its members have the same color, we look for a structure where no two elements have the same color.

An {\em r-coloring} of a set $S$ is a map from $S$ to $\{1,2,\ldots,r\}$. One of the earliest results in Ramsey theory is due to Schur \cite{Schur}, which states that for any positive integer $r$, there exists a least positive integer $S(r)$ such that every $r$-coloring of  $\{1,2,\ldots ,S(r)\}$ contains a monochromatic solution to the equation $x_1 + x_2 = x_3$. More generally,  for each $m \geq 3$, one can consider the Schur-type numbers for the equation
$$E_m:   x_1 + x_2 + \cdots  + x_{m-1} = x_m; $$
the existence of these numbers  follows easily from the work of Rado \cite{R}. In this paper, we consider solutions to $E_m$ from the perspective of rainbow Ramsey theory.

If $\chi$ is an $r$-coloring of a set $S$ that uses all $r$ colors (i.e, it is surjective), we say that $\chi$ is an {\em exact} $r$-coloring.  For $m\ge 3$ and $n\ge \frac{m(m-1)}{2}$, define the {\it rainbow Schur number} $RS_m(n)$ to be the minimum number of colors such that every exact $RS_m(n)$-coloring of $[1,n]$ contains a rainbow solution to $E_m$.  The assumption that $n\ge \frac{m(m-1)}{2}$ guarantees that $E_m$ is solvable using distinct elements from $[1,n]$. Since coloring $[1,n]$ using $n$ colors will always produce a
rainbow solution to $E_m$, we know that $RS_m(n)\le n$.  Note that rather than fix the number of colors and seek an optimal value of $n$, as in the definition of Schur numbers, rainbow Schur numbers fall under the subject of anti-Ramsey theory (introduced by Erd\H{o}s, Simonovits, and S\'os \cite{ESS}) by fixing $n$ and seeking an optimal number of colors.

In \cite{B}, it was proved that \begin{equation}RS_3(n)=\floor{\log _2(n)}+2,\quad \mbox{for all} \ n\ge 3.\notag
\end{equation} This result was proved independently by Fallon et al. \cite{FGRWW}, where it was also shown that \begin{equation} RS_4(n)=\ceil{\frac{n+6}{2}}, \quad \mbox{for all} \ n\ge 6.\label{Fallon4}\end{equation} In addition,  \cite{FGRWW}  provides a general lower bound for $R_{m}(n)$ by exhibiting a particular coloring of $[1,n]$ that avoids rainbow solutions to $E_m$.    In Section \ref{main},  we we determine the exact value of $RS_m(n)$  for all $m \geq 4$ and $n \geq \frac{m(m-1)}{2}$, which has Equation (\ref{Fallon4}) as a special case. In Section \ref{weak}, we provide  a formula for a generalization of rainbow Schur numbers, which we call weakened rainbow Schur numbers, where, for a given $t \le m$, we seek solutions to $E_m$ that use at least $t$ colors.

Results involving  rainbow Ramsey numbers for 3-term and 4-term arithmetic progressions may be found in \cite{JNR}. In \cite{DLMOR}, the authors provide a rainbow version of Rado's work on systems of linear homogeneous equations. Work on rainbow solutions to 3-variable linear equations in the group $\mathbb{Z}_{n}$ appears in  \cite{AEHNWY}, \cite{BKKTTY}, \cite{H}, \cite{HM}, and \cite{LM}.
Related work arises in \cite{Con}, \cite{FJR}, and \cite{FMR}, where additional restrictions are placed on the number of times each color occurs.
The work in \cite{FMMRWZ} deals with rainbow solutions to $E_3$ in the rectangular grid $[1,m] \times [1,n]$, with coordinate-wise addition.
As far as we know, weakened (as defined above) versions of the results contained in these related papers have not yet been studied.

\section{The Value of $RS_m(n)$} \label{main}

In this section, we give the exact value of $RS_m(n)$ for all $m \geq 4$ and $n \geq \frac{m(m-1)}{2}$. The next theorem shows that this value serves as a lower bound for $RS_m(n)$. In  \cite{FGRWW}, the authors give a proof of this lower bound, but for the sake of completeness we include a proof here, which is similar to, but slightly different from their proof.

\begin{theorem}\label{rainbowlower}
Let $m\ge 4$ and let $n \geq \frac{m(m-1)}{2}$.  Then
\[RS_m(n)\ge \left \lceil \frac{(m-3)n+\frac{m(m-1)}{2}}{m-2}\right \rceil.\]
\end{theorem}

\begin{proof}
Suppose that $a_1+a_2+\cdots +a_{m-1}=a_m$ is any solution to $E_m$ that lies within $[1,n]$, where $a_i < a_{i+1}$ for all $1\le i\le m-2$.  Let
$$k= n + 2 - \left\lceil \frac{(m-3)n+\frac{m(m-1)}{2}}{m-2}\right\rceil.$$
We claim that $a_2\le k$.  For a contradiction,  assume $a_{2} > k$. Then
\begin{align}
   \sum_{i=1}^{m-1} a_{i} &\geq   1+\sum_{i=1}^{m-2} (k+i)  \notag \\
                          &= 1+ (m-2)(n+2) - (m-2)\left\lceil \frac{(m-3)n+\frac{m(m-1)}{2}}{m-2}\right\rceil \notag \\
                          &\qquad +\frac{(m-1)(m-2)}{2}  \notag \\
                          & > 1+(m-2)(n+2)  - (m-2) \frac{(m-3)n+ \frac{m(m-1)}{2}}{m-2} - (m-2) \notag \\
                          &\qquad + \frac{(m-1)(m-2)}{2}  \notag \\
                          & = 1 + n + (m-2)\left(2 - 1 + \frac{m-1}{2}\right) - \frac{m(m-1)}{2} \notag \\
                          & = 1+n+m-2+\frac{m-1}{2}(-2) =  n, \notag
\end{align}
which contradicts the assumption that the solution lies in $[1,n]$.

To complete the proof, color $[1,n]$ as follows.  Color all members of $[1,k]$ the same color, say red, and color each element of $[k+1, n]$ its own unique color, different from red.  Note that this is an exact $(n-k+1)$-coloring, that is,
it uses
 $$ \left\lceil \frac{(m-3)n+\frac{m(m-1)}{2}}{m-2}\right\rceil  -1 $$ colors.  By the above claim, $a_1$ and $a_2$ are both colored red, and hence, the solution is not rainbow.  So, at least
 $$ \left\lceil \frac{(m-3)n+\frac{m(m-1)}{2}}{m-2}\right \rceil $$ colors are needed to guarantee a rainbow solution to $E_m$.
\end{proof}

Before presenting the next theorem, we give a definition.

\begin{definition}  Let $n \geq r$ be positive integers, and let $\chi$ be an exact $r$-coloring of $[1,n]$. An integer $x \in [1,n]$ is called a {\em surplus integer} of $\chi$ if there exists a $y < x$ such that $\chi(x) = \chi(y)$.
\end{definition}
\begin{remark}\label{surplus}  Let $s_{\chi}$ denote the number of surplus integers of $\chi$. From the definition, we see that, for any exact $r$-coloring $\chi$ of an interval $[1,n]$, we have $s_\chi = n - r$.
\end{remark}

\begin{theorem}\label{GeneralUpperBound}
Let $m, n \in \mathbb{Z^{+}}$ with $m \geq 4$ and $n \geq \frac{m(m-1)}{2}$.  Assume that one of the following conditions hold:
\begin{equation*}  m \mbox{ is odd and }  n \equiv 1\Mod{(m-2)}
\end{equation*}
or
\begin{equation*} m \mbox{ is even and } n \equiv \frac{m}{2} \Mod{(m-2)}.
\end{equation*}
Let
\[  c(n,m) = \frac{(m-3)n+m(m-1)/2}{m-2}.\]
Then every exact $c(n,m)$-coloring of $[1,n]$ contains a rainbow solution to $E_m$.
\end{theorem}

\begin{proof} First note that the assumptions regarding $m$ and $n$ imply that $c(n,m)$ is a positive integer for all such $m$ and $n$. We begin with the case in which $m$ is odd. We use induction on $\ell\geq \frac{m+1}{2}$, where $n=(m-2)\ell+1 $.  If $\ell=\frac{m+1}{2}$, then we have
\[ n = (m-2)\frac{m+1}{2} + 1 = \frac{m(m-1)}{2} \]
and
$$c(n,m)  =   \frac{(m-3)\frac{m(m-1)}{2}+\frac{m(m-1)}{2}}{m-2}  =\frac{m(m-1)}{2}.$$
Hence, since $c(n,m) = n = \frac{m(m-1)}{2}$, there is only one exact $c(n,m)$-coloring of $[1,n]$, and it contains the rainbow equation $1+2+ \cdots + (m-1) = n$. Therefore, the result is true when   $\ell=\frac{m+1}{2}$.

We now let $\ell \geq \frac{m+1}{2}$ and assume the result holds for $n=(m-2)\ell+1$.
To complete the proof, we will show that the result holds for $(m-2)(\ell + 1) + 1 =n+m-2$, i.e., that every exact $c(n+m-2,m)$-coloring of $[1,n+m-2]$ has a rainbow solution to $E_m$. For a contradiction, assume $\chi$ is an exact $c(n+m-2,m)$-coloring of $[1,n+m-2]$ with no such rainbow solution.  By the inductive hypothesis, we may assume that no more than $c(n,m)-1 = c(n+m-2,m) - (m-2)$ colors are used in $[1,n]$.
This implies that each member of $[n+1,n+m-2]$ is the only member of its color class under $\chi$.

For each $j$ such that $0 \leq j \leq m-4$, we define $t_{j}$ recursively as follows. Let $t_{0} = 1$.  For each $j=1,2,\ldots, m-4$, define $t_{j}$ be the least integer  greater than $t_{j-1}$ such that
 $\chi(t_j) \not\in \{\chi(t_i): 0 \leq i \leq j-1\}$. From this definition, we see that the number of surplus integers of $\chi$ that lie within $[1,t_{m-4}]$ is $t_{m-4} - (m-3)$, so that (using the notation of Remark \ref{surplus})
 \begin{equation}\label{Equation 2} t_{m-4} \leq s_{\chi} + m - 3.
 \end{equation}
Note that, by Remark \ref{surplus} and the definition of $\chi$,
\begin{align}
 s_\chi  &=  n+m-2 - c(n+m-2,m) \notag \\ &= n+m-2 - \frac{(m-3)(n+m-2) + \frac{m(m-1)}{2}}{m-2}.
\label{surplusEm}\end{align}
By  Equations (\ref{Equation 2}) and (\ref{surplusEm}), we have
 \begin{align}
 t_{m-4} & \leq s_{\chi} + m - 3 \notag \\
  & = n+2m-5 - \frac{(m-3)(n+m-2) + \frac{m(m-1)}{2}}{m-2}  \notag \\
  & =  \frac{n}{m-2} + 2m-5 - \frac{((m-3)(m-2) + \frac{m(m-1)}{2}}{m-2} \notag \\
  & = \frac{n}{m-2} + m-2 - \frac{\frac{m(m-1)}{2}}{m-2}. \notag
 \end{align}
Since the $t_j$ are strictly increasing, it follows that
 \[ t_{m-5} \leq \frac{n}{m-2} + m-3 - \frac{\frac{m(m-1)}{2}}{m-2}\]
 and, more generally,
 \begin{equation}\label{tjBound}
  t_j \leq \frac{n}{m-2} + j+2 - \frac{\frac{m(m-1)}{2}}{m-2},
 \end{equation}
  for $1 \leq j \leq m-4$.

 Let
 \[v = \left\lfloor \frac{n-3t_{m-4}- \sum_{j=1}^{m-5}t_{j} + m -4}{2}\right\rfloor. \]
 Note that $v \geq 1$ since, using Inequality (\ref{tjBound}),
 \begin{align} n-3t_{m-4}- \sum_{j=1}^{m-5}t_{j} + m -4 & \geq   n - \frac{3n}{m-2} - 3\left( m-2 - \frac{\frac{m(m-1)}{2}}{m-2}\right)+m-4 \notag \\
 &\quad   - (m-5)\frac{n}{m-2}  - \sum_{j=1}^{m-5}\left(j+2 -  \frac{\frac{m(m-1)}{2}}{m-2}\right) \notag \\
 & = -4m+12 + \frac{3 \frac{m(m-1)}{2}}{m-2} - \frac{(m-5)(m-4)}{2} \notag \\
 &\quad + \frac{(m-5)\frac{m(m-1)}{2}}{m-2} \notag\\
 & = -4m+12 + \frac{m(m-1)}{2} - \frac{(m-5)(m-4)}{2} \notag \\
 & =  2.\notag
 \end{align}
 For each $i$ such that  $1 \leq i \leq v$, let \[a_i = t_{m-4}+i\] and \[b_i = n+m-3-2t_{m-4}-\sum_{j=1}^{m-5}t_{j}  - i,\] and let $P_i =  \{a_i,b_i\}$.
 Note that the $a_i$'s are strictly increasing,  the $b_{i}$'s are strictly decreasing, and that $\max\{a_i\}$ and $\min\{b_i\}$
both occur when $i = v$.  Also, we have that $a_v < b_v$, since
  \begin{align}
  b_v - a_v   n-2t_{m-4}-\sum_{j=1}^{m-5}t_{j} +m&-3-v - (t_{m-4}+v)  \notag \\
  & = n-3t_{m-4}-\sum_{j=1}^{m-5}t_{j}+m-3 -2v \notag \\
  & > 0.\notag
  \end{align}
From these facts, we see that the sets $P_i$ are pairwise disjoint and that $|P_i| = 2$ for each $i$.

If there is  some pair $P_u = \{a_u,b_u\}$ such that $\chi(a_u) \neq \chi(b_u)$ and
\[ \{\chi(a_u), \chi(b_u)\} \cap \{\chi(t_i): 0 \leq i \leq m-4\} = \emptyset , \]
then  $E_m$ has the rainbow solution
\[ 1+t_1 + t_2 + \cdots t_{m-4} +(t_{m-4} + u) + \left(n+m-3-2t_{m-4}- \sum_{j=1}^{m-5}t_{j}- u\right) = n+m-2 \]
since, as noted previously,  $n+m-2$ is the only member of its color class.
 Hence, by our assumption about $\chi$,  no such $P_u$ exists.  Thus, each $P_i$ contributes at least one surplus integer to $s_\chi$.
This implies that
\begin{equation}\label{SurplusBoundwithv}
s_{\chi} \geq t_{m-4}-(m-3)+v,
\end{equation}
because  there are exactly $t_{m}- (m-3)$ surplus integers contained in $[1,t_{m-4}]$.

By Inequality (\ref{tjBound}), we have
\begin{align}
 \sum_{j=1}^{m-4}t_{j} & \leq (m-4)\left( \frac{n-\frac{m(m-1)}{2}}{m-2}+2\right) + \frac{(m-4)(m-3)}{2}\notag \\
                       & =   (m-4) \frac{2n - m(m-1)+4(m-2) + (m-3)(m-2)}{2(m-2)}\notag \\
                       & =   (m-4)\frac{2n-2}{2(m-2)} \notag \\
                       & =    \frac{(m-4)(n-1)}{m-2}.\label{tjSum}
 \end{align}
Using Inequalities (\ref{SurplusBoundwithv}) and  (\ref{tjSum}), and the definition of $v$, we have
\begin{align}
 s_\chi & \geq   t_{m-4}-(m-3) + v   \notag\\
  & \geq t_{m-4} - (m-3) +  \frac{n-3t_{m-4}- \sum_{j=1}^{m-5}t_{j} + m -5}{2} \notag\\
  & = \frac{n}{2}-\frac{\sum_{j=1}^{m-4}t_{j}}{2}  - \frac{m-1}{2} \notag\\
  & \geq \frac{n}{2} - \frac{(m-4)(n-1)}{2(m-2)} - \frac{m-1}{2} \notag\\
  & = n\left(\frac{1}{2}-\frac{m-4}{2(m-2)}\right) + \frac{m-4}{2(m-2)} - \frac{m-1}{2} \notag\\
  & = \frac{n}{m-2} - \frac{m^{2}-4m+6}{2(m-2)}.\label{LowerBoundonschi}
\end{align}
Now, from Equation (\ref{surplusEm}) we have
\begin{align}
s_{\chi} & = n+m-2 - \frac{(m-3)(n+m-2) + \frac{m(m-1)}{2}}{m-2} \notag\\
& = \frac{n}{m-2} -\frac{2(m-2)^{2}-3m^{2}+11m-11}{2(m-2)}\notag\\
& = \frac{n}{m-2} - \frac{m^{2}-3m+4}{2(m-2)},\notag
\end{align}
which is in contradiction to  Inequality (\ref{LowerBoundonschi}).  This completes the proof for the case in which $m$ is odd.

The proof for $m$ even is almost identical to the proof for $m$ odd. The only difference is that the induction is done on $\ell$ where $n = (m-2)\ell + \frac{m}{2}$,  and the initial value of $\ell$ is taken to be $\frac{m}{2}$. Then, just as in the odd case, this initial value of $\ell$ again gives
     \[ n = \frac{m(m-1)}{2} = c(n,m). \]
Hence, as explained in the case of $m$ odd,  for the initial step of the induction for $m$ even, we have that the result holds for $\ell = \frac{m}{2}$.  For the inductive step, we assume that the result holds for $n = \ell (m-2) + \frac{m}{2}$ for some $\ell \geq  \frac{m}{2}$, and must then show that it holds for $(\ell+1)(m-2) + \frac{m}{2}$. The rest of the proof is the same as that for the odd case.
\end{proof}

We are now able to give the exact value of $R_m(n)$ when $m \geq 4$.

\begin{theorem} \label{mOddAlln} Let $m \geq 4$  and $n \geq \frac{m(m-1)}{2}$.  Then
\begin{equation}\label{Formula}
RS_{m}(n) =  \left\lceil \frac{(m-3)n+\frac{m(m-1)}{2}}{m-2} \right\rceil.
\end{equation}
\end{theorem}
 \begin{proof}  Let $m \geq 4$ and $n \geq \frac{m(m-1)}{2}$.
 Let $a(n,m)$ denote the right-hand side of Equation (\ref{Formula}). By Theorem \ref{rainbowlower}, we know that $a(n,m)$ is  a lower bound for $RS_{m}(n)$.
We claim that $a(n,m)$ is also an upper bound.  We begin with the case in which
\[ n = \frac{m(m-1)}{2} + i, \]
 such that $0 \leq i \leq m-4$.  In this case, we have
\begin{align}
a(n,m)& = \left\lceil \frac{(m-2)\frac{m(m-1)}{2}+i(m-3)}{m-2} \right\rceil \notag \\
& = \frac{m(m-1)}{2} + \left\lceil i - \frac{1}{m-2} \right\rceil \notag \\
& = n.\notag
\end{align}
Since the only exact $n$-coloring of $[1,n]$ has no two elements with the same color, we have that $1+2+ \cdots + (m-1) = \frac{m(m-1)}{2}$ is a rainbow solution to $E_m$. This shows that, in this case, $a(n,m)$ is an upper bound on $RS_m(n)$.
Thus, we will assume that \[ n \geq \frac{m(m-1)}{2} + m - 3 .\]

Consider the case in which $m$ is odd.
By Theorem \ref{GeneralUpperBound}, we know that the claim is true whenever $n \equiv 1 \Mod{(m-2)},$ so we may assume
   that $n \equiv  i \Mod{(m-2)}$ where  $2 \leq i \leq m-2$.
  Since $m$ is odd,   $\frac{m(m-1)}{2} \equiv 1 \Mod{(m-2)}$, and therefore
\begin{equation}\label{aNumerator}
 (m-3)n+ \frac{m(m-1)}{2} \equiv -i+1 \Mod{(m-2)}.
 \end{equation}
Let $\chi$ be any exact $a(n,m)$-coloring of $[1,n]$.
  Since $1 \leq i - 1 \leq m-3$, from Equation (\ref{aNumerator}), we have
 \begin{equation} \label{aEquation}
  a(n,m) =    \frac{(m-3)n+\frac{m(m-1)}{2}+i-1}{m-2}.
  \end{equation}
 From Equation (\ref{aEquation}), within the interval $[1,n-i+1]$, there must be at least
\begin{align}
a(n,m)- i + 1 & = \frac{(m-3)n + \frac{m(m-1)}{2}+i-1 + (m-2)(-i+1)}{m-2}\notag \\
     & = \frac{(m-3)(n-i+1)+ \frac{m(m-1)}{2}}{m-2} \label{ColorsShorterIntervalOdd}
\end{align}
different colors.

By Theorem \ref{GeneralUpperBound}, since $n-i+1 \equiv 1 \Mod{(m-2)}$  and $n-i+1 \geq \frac{m(m-1)}{2}$, it follows that
  \[RS_{m}(n-i+1) \leq \frac{(m-3)(n-i+1)+ \frac{m(m-1)}{2}}{m-2}.\]
  Hence, from Equation (\ref{ColorsShorterIntervalOdd}), under $\chi$ there is a rainbow solution to $E_m$ within $[1,a(n,m)-i+1]$, and therefore within $[1,a(n,m)]$, which completes the proof for odd values of $m$.

  Now assume that $m$ is even. Similar to the proof of the odd case,
   we may assume that $n \equiv  i \Mod{(m-2)}$ where  $\frac{m}{2}+1 \leq i \leq \frac{m}{2}+m-3$.
   Since $m$ is even, we have $\frac{m(m-1)}{2} \equiv \frac{m}{2}\Mod{(m-2)}$, so that
\begin{equation*}\label{aNumeratorEven}
 (m-3)n+ \frac{m(m-1)}{2} \equiv -i+\frac{m}{2} \Mod{(m-2)}.
 \end{equation*}
From this and the fact that  $1 \leq i - \frac{m}{2} \leq m-3$, it follows that
 \begin{equation} 
  a(n,m) =    \frac{(m-3)n+\frac{m(m-1)}{2}+i-\frac{m}{2}}{m-2}.\notag
  \end{equation}
  Like in the odd case, if $\chi$ is any exact $a(n,m)$-coloring of $[1,n]$,
   then within the interval $[1,n-i+\frac{m}{2}]$ there must be at least
 \begin{equation}\label{ColorsShorterInterval}
 a(n,m)- i + \frac{m}{2} = \frac{(m-3)(n-i+\frac{m}{2})+ \frac{m(m-1)}{2}}{m-2}
 \end{equation}
 different colors. By Theorem \ref{GeneralUpperBound}, since $n-i+\frac{m}{2} \equiv \frac{m}{2} \Mod{(m-2)}$, it follows that
  \[RS_{m}(n-i+\frac{m}{2}) \leq \frac{(m-3)(n-i+\frac{m}{2})+ \frac{m(m-1)}{2}}{m-2}.\]
  Hence, from Equation (\ref{ColorsShorterInterval}), under $\chi$ there is a rainbow solution to $E_m$ within $[1,a(n,m)-i+\frac{m}{2}]$, and therefore within $[1,a(n,m)]$, which completes the proof.
  \end{proof}

\section{Weakened Rainbow Schur Numbers}\label{weak}

For $m\ge 3$ and $2\le t\le m$, define the {\it weakened rainbow Schur number} $RS_{t,m}(n)$ to be the minimum number of colors such that every exact $RS_{t,m}(n)$-coloring of $[1,n]$ contains a solution to $E_m$ that uses at least $t$ of the colors.  Here, $RS_{m,m}(n)$ agrees with the rainbow Schur number $RS_m(n)$. Note that if $t < m$ then, in contrast to the situation with rainbow colorings of $E_m$, the relevant solutions to $E_m$ do not necessarily consist of distinct summands.

When $t_1 \le t_2$, observe that every solution to $E_m$ that uses at least $t_2$ colors necessarily uses at least $t_1$ colors.  It follows that $$RS_{t_1,m}(n)\le RS_{t_2,m}(n),$$ for all $m\ge 3$ and $n$ for which a solution to $E_m$ exists that can use at least $t_2$ colors.  From this observation, we see that $RS_{t,m}(n)$ is defined for all $t$ such that $2\le t<m$ whenever $RS_m(n)$ is defined.

While a range of values for $n$ was not specified in the definition of $R_{t,m}$, it is natural to only consider values of $n$ for which there can exist a solution to $E_m$ that uses at least $t$ colors. For any $m\ge 3$ and $2\le t\le m$, the equation
$$\underbrace{1+1+\cdots +1}_{m-t+1\ terms}+2+3+\cdots +(t-1)=\frac{t(t-1)}{2}+m-t$$ has the least sum among all equations in $E_m$ that can be colored using at least $t$ colors.  For this reason, we assume $n\ge \frac{t(t-1)}{2}+m-t$ when considering $RS_{t,m}(n)$.  In the evaluations of $RS_{t,m}(n)$ that follow, we often restrict the values of $n$ beyond this natural bound.

\begin{theorem}
For all $m\ge 3$ and $n\ge 2m-4$, we have $RS_{2,m}(n)= 2$.
\end{theorem}

\begin{proof}
At least two colors are required in order to have a $2$-colored solution to $E_3$, and hence $RS_{2,3}(n)\ge 2$.  Now consider an exact $2$-coloring of $[1,n]$.  Without loss of generality, assume that $1$ is red and $i\in [2,n]$ is the least positive integer that is colored blue.  We consider two cases, based on the value of $i$.

\underline{Case 1} If $i\le n-m+2$, then the equation $$\underbrace{1+1+\cdots +1}_{m-2\ terms}+i=i+m-2\le n$$ is in $E_m$ and uses at least two colors.

\underline{Case 2} If $i> n-m+2$, then consider the equation $$\underbrace{1+1+\cdots +1}_{m-2\ terms}+(i-(m-2))=i.$$  This equation uses at least two colors and is in $E_m$ whenever $i\ge m-1$, which for this case, occurs when $$n-m+3\ge m-1.$$ This is equivalent to $n\ge 2m-4$, as assumed in the statement of the theorem.

In both cases, we find that there exists an equation in $E_m$ that uses two colors, from which it follows that $RS_{2,m}(n)\le 2$.
\end{proof}

To demonstrate the need for the assumption $n\ge 2m-4$ in the previous theorem, consider the case where $m=6$, $t=2$, and $n=6$.  The only solutions to $E_6$ contained in $[1,n]$  are $$1+1+1+1+1+=5 \quad \mbox{and} \quad 1+1+1+1+2=6.$$  Using the color classes $$C_1=\{1, 2, 5, 6\}, \quad C_2=\{3\}, \quad \mbox{and} \quad C_3=\{4\},$$ we find that $RS_{2,6}(6)\ge 4$.

\begin{theorem}
Let  $m\ge 4$ and $3\le t\le m$. Then for all  $n\ge \frac{t(t-1)}{2}+m-t$, we have $$RS_{t,m}(n)= \ceil{\frac{(t-3)n+\frac{t(t-1)}{2}+m-t}{t-2}}.$$
\end{theorem}

\begin{proof}
Let $$k= k(n,m,t)= \ceil{\frac{(t-3)n+\frac{t(t-1)}{2}+m-t}{t-2}}$$ and note that
\begin{equation}
k(t-2)\le (t-3)(n)+\frac{t(t-1)}{2}+m-3.\label{weakineq2}
\end{equation}
Let $\alpha$ be the exact $(k-1)$-coloring of $[1,n]$ having the following color classes:
$$C_1=[1, n+2-k], \ C_2=\{n+3-k\}, \ C_3=\{n+4-k\},\ \dots, \ C_{k-1}=\{n\}.$$
To show that $RS_{t,m}(n)\ge k$, it suffices to show that $\alpha$ does not have a solution to $E_m$ that uses at least $t$ colors. For a contradiction, assume that $$a_1 + a_2 + \cdots a_{m-1} = a_m, \quad  \mbox{where} \ a_1 \leq a_2 \leq \cdots \leq a_{m-1},$$ is such a solution.
Since at least $t$ colors occur among the $a_i$, no color class contains more than $m-t+1$ of the $a_i$. Hence, by Inequality (\ref{weakineq2}),
\begin{align}  a_{1} +  a_{2}  +\cdots &+a_{m-1} \notag \\
& \leq \underbrace{1+1+\cdots +1}_{m-t+1\ terms}+(n+3-k)+(n+4-k)+\cdots +(n+t-k) \notag \\
&=(m-t+1)+3+4+\cdots +t+n(t-2)-k(t-2) \notag \\
&=(m-t-2)+\frac{t(t+1)}{2}+n(t-2)-k(t-2) \notag \\
&\ge (m-t-2)+\frac{t(t+1)}{2}+n(t-2) \notag \\
&\qquad \qquad \qquad -\left( (t-3)n+\frac{t(t-1)}{2}+m-3\right)\notag \\
&=n+1, \notag \end{align}
which contradicts the fact that $a_m \leq n$.

To prove that
\begin{equation}\label{Theorem 7 UpperBd} RS_{t,m} (n) \leq k(n,m,t),
\end{equation}
we use induction on $m+n$, where $m\ge 4$ and $n\ge \frac{t(t-1)}{2}+m-t$. To establish the base cases of the induction, we will show that the Inequality (\ref{Theorem 7 UpperBd}) holds for each of the following two cases: (a) all  $n \geq \frac{t(t-1)}{2}+m-t$, when $m=4$ and $3 \leq t \leq 4$; and (b) $n = \frac{t(t-1)}{2}+m-t$, for all $m\geq 4$ and $3 \leq t \leq m$.

To establish Inequality (\ref{Theorem 7 UpperBd}) in case (a),  note that for $m=4$, we have either $t=3$ or $t=4$. When $m=t=4$, Inequality (\ref{Theorem 7 UpperBd}) holds by Equation (\ref{Fallon4}).
Now assume that $m=4$ and $t=3$.  We will show that  $RS_{3,4}(n)\le k(n,4,3)=4$ by induction on $n\ge 4$.  When $n=4$, an exact $4$-coloring of $[1,4]$ has every number receiving a unique color, and hence  $1+1+2=4$ is a solution to $E_4$ that uses (at least) $3$ colors.  Now assume that $RS_{3,4}(n-1)\le 4$ for some $n-1 \ge 4$ and let $\beta$ be an exact $4$-coloring of $[1,n]$. If $n$ is a surplus integer under $\beta$, then $[1,n-1]$ uses all $4$ colors and, since $RS_{3,4}(n-1)\le 4$, it contains   a solution to $E_4$ that uses at least $3$ colors.

If $n$ is not a surplus integer under $\beta$,  then it receives its own unique color and $[1,n-1]$ uses $3$ colors. In this latter situation,  let $i\in [2,n]$ be the least integer such that $\beta (i)\ne \beta (1)$, and  consider the equation \begin{equation} 1+i+(n-i-1)=n.\label{basecase3,4}\end{equation}  Note that $n-i-1$ is a positive integer, since otherwise $i\ge n-1$, which would imply that $\beta$ uses at least $5$ colors.  It follows that Equation (\ref{basecase3,4}) is a solution to $E_4$ that uses at least $3$ colors.

To show that Inequality (\ref{Theorem 7 UpperBd}) holds in case (b),  we have
\begin{align}
k(n,m,t) &= \ceil{\frac{(t-3)\left(\frac{t(t-1)}{2}+m-t\right)+\frac{t(t-1)}{2}+m-t}{t-2} } \notag \\
&=\ceil{\frac{(t-2)\left(\frac{t(t-1)}{2}+m-t\right)}{t-2}} \notag \\ &=\frac{t(t-1)}{2}+m-t = n, \notag
\end{align}
and hence, in any exact $\left(\frac{t(t-1)}{2}+m-t\right)$-coloring of $[1,n]$, each element is the only member of its color class.  So, in this case, every solution to $E_m$ is necessarily rainbow, establishing  Inequality (\ref{Theorem 7 UpperBd})  in this case.

Having taken care of the base cases in the inductive proof of Inequality (\ref{Theorem 7 UpperBd}),
we now let $m \geq 5$,  $t \geq 3$, and $n\ge \frac{t(t-1)}{2}+m-t+1$.   Assume that for all $m' \geq 4$ and $n' \geq \frac{t(t-1)}{2}+m-t$ with $m'+n' <m+n$, we have
\begin{equation}\label{IndHyp}
 RS_{t,m'}(n') \leq k(n',m',t)
 \end{equation}
 for all $3\le t\le m'$.  To complete the proof, it suffices to prove that for every $t$ such that $3 \leq t \leq m$, every exact $k(n,m,t)$-coloring of $[1,n]$ contains a solution to $E_m$ using at least $t$ colors. Let $\chi$ be an exact $k(n,m,t)$-coloring of $[1,n]$.
By the Division Algorithm, let
 \begin{equation}\label{DA}
 (t-3)n+\frac{t(t-1)}{2}+m-t=(t-2)\ell +j,
 \end{equation}
  where $\ell, j\in \mathbb{Z}$ and $0\le j\le t-3$.  The remainder of the proof is separated into two cases.

\underline{Case 1} Assume that $1\le j\le t-3$.  By Equation (\ref{DA}),
$$ k(n,m,t) = \ceil{\frac{(t-3)n+\frac{t(t-1)}{2}+m-t}{t-2}}=\ceil{\ell +\frac{j}{t-2}}=\ell +1.$$
Therefore, the interval $[1,n-1]$ uses at least $\ell$ colors under $\chi$.
Now, by Equation (\ref{IndHyp}), \begin{align} RS_{t,m}(n-1)&\le k(n-1,m,t)\notag \\ &=\ceil{\frac{(t-3)n+\frac{t(t-1)}{2}+m-t}{t-2}-\frac{t-3}{t-2}}\notag \\&=\ceil{\ell-\frac{(t-3)-j}{t-2}}=\ell.
\notag \end{align}   Therefore, in $[1,n-1]$ (and hence in $[1,n]$) there is a solution to $E_m$ that uses at least $t$ colors.

\underline{Case 2} Assume that $j=0$.  Then $k(n,m,t) = \ell$  and
$$k(n-1,m,t) = \ceil{\frac{(t-3)(n-1)+\frac{t(t-1)}{2}+m-t}{t-2}}= \ceil{\ell -\frac{t-3}{t-2}}=\ell.$$  If  $[1,n-1]$ uses all $\ell$ colors of $\chi$, then by Equation (\ref{IndHyp}) there exists a solution to $E_m$ that uses at least $t$ colors.  Otherwise, $[1,n-1]$ uses only $\ell-1$ colors and the integer $n$ is the only member of its color class.  We may assume that $t\le m-1$ since the $t=m$ case corresponds with Theorem \ref{mOddAlln}.  For $3\le t\le m-1$, by Equation (\ref{IndHyp}) we obtain
\begin{align}
RS_{t,m-1}(n-1)&\le k(n-1,m-1,t) \notag \\
&=\ceil{\frac{(t-3)(n-1)+\frac{t(t-1)}{2}+(m-t-1)}{t-2}} \notag \\
&=\ceil{\ell -\frac{t-3}{t-2}-\frac{1}{t-2}}\notag \\&=\ell -1.\notag
\end{align}
It follows that there exists a solution $a_1+a_2+\cdots + a_{m-2}=a_{m-1}$ to $E_{m-1}$ that uses at least $t$ colors, where $a_{m-1}\le n-1$.  Then $$a_1+a_2+\cdots + a_{m-2}+(n-(a_1+a_2+\cdots +a_{m-2}))=n$$ is a solution to $E_m$ that uses at least $t$ colors.

In both cases, $\chi$ contains a solution to $E_m$ that uses at least $t$ colors.  It follows that $RS_{t,m}(n)\le k(n,m,t)$, completing the proof.
\end{proof}

\bibliographystyle{amsplain}

\end{document}